\documentclass[12pt]{article}

\usepackage{amssymb}
\usepackage{amsmath}
\usepackage{amsbsy}
\usepackage{amscd}
\usepackage{amsfonts}
\usepackage{amsthm}
\usepackage{mathrsfs}
\usepackage{verbatim}
\usepackage[colorlinks]{hyperref}
\usepackage{fullpage}
\usepackage{mathdots}
\usepackage{graphicx,subfigure}
\usepackage[english]{babel}
\usepackage[utf8]{inputenc}
\usepackage{tikz}
\usepackage{stackrel}
\usepackage{authblk}

\usepackage{tikz-cd}

\usepackage{mathtext}

\usetikzlibrary{backgrounds,fit, matrix}
\usetikzlibrary{positioning}
\usetikzlibrary{calc,through,chains}
\usetikzlibrary{arrows,shapes,snakes,automata, petri}

\usepackage{booktabs}
\usepackage{adjustbox}

\newtheorem{Theorem}[equation]{Theorem}

\newtheorem{Lemma}[equation]{Lemma}
\newtheorem{Proposition}[equation]{Proposition}

\theoremstyle{definition}

\newtheorem{Remark}[equation]{Remark}

\numberwithin{equation}{section}
\numberwithin{figure}{section}

\newcommand{\C}{{\mathbb C}}
\newcommand{\Z}{{\mathbb Z}}
\newcommand{\Q}{{\mathbb Q}}

\newcommand{\N}{{\mathbb N}}

\newcommand{\mt}[1]{\text{#1}}

\newcommand{\mbf}[1]{\mathbf{#1}}

\begin{document}

\title{A Filtration on Equivariant Borel-Moore Homology}

\author{ARAM BINGHAM, MAHIR BILEN CAN, YILDIRAY OZAN}

\maketitle

\begin{abstract}

Let $G/H$ be a homogeneous variety, 
and let $X$ be a $G$-equivariant embedding of $G/H$
such that the number of $G$-orbits in $X$ is finite. 
We show that the equivariant Borel-Moore homology of $X$ 
has a filtration with associated graded module the direct sum of the 
equivariant Borel-Moore homologies of the $G$-orbits. 
If $T$ is a maximal torus of $G$ such that each $G$-orbit 
has a $T$-fixed point, then the equivariant filtration 
descends to give a filtration on the ordinary Borel-Moore homology 
of $X$. We apply our findings to certain wonderful compactifications 
as well as to double flag varieties. 

\vspace{.5cm}


\end{abstract}

\normalsize

\section{Introduction}

Let $G$ be a connected complex algebraic group, and let $K$ be a closed subgroup.
A $G$-variety $X$ is said to be a {\em $G$-equivariant embedding of $G/K$} if there is 
an open $G$-orbit $O$ in $X$ such that the stabilizer subgroup of a point from $O$ is isomorphic to $K$. 
In this paper we are concerned with the (equivariant) Borel-Moore homology groups of $X$ under the assumption 
that $X$ contains only finitely many $G$-orbits. 
Let us denote by $A^G_*(X)_\Q$ the rational $G$-equivariant Chow group of $X$,
and let us denote by $H^G_*(X,\Q)$ the rational $G$-equivariant Borel-Moore homology group of $X$. 
Then there is a degree doubling cycle map $cl^G_X : A^G_*(X)_\Q \rightarrow H^G_*(X,\Q)$,
see~\cite{EdidinGraham}.
In this notation, the main general result of our paper is the following.

\begin{Theorem}\label{I:T1}
Let $G$ be a complex connected linear algebraic group, and let $X$ 
be a complex $G$-variety. 
If $X$ is comprised of finitely many $G$-orbits, then 
$cl^G_X$ is an isomorphism between $A^G_*(X)_\Q$ and $H^G_*(X,\Q)$. 
Furthermore, as a $H^*_G(pt,\Q)$-module, the equivariant Borel-Moore homology 
has a decomposition, 
$
H^G_*(X,\Q) \cong H^G_*(Y,\Q) \oplus H^G_*(U,\Q),
$
where $Y$ is any closed orbit in $X$, and $U$ is the complement of $Y $ in $X$. 
Consequently, the equivariant Borel-Moore homology of $X$ has a 
filtration, as a graded module over $H^*_G(pt,\Q)$, with associated graded module equal to 
the direct sum of equivariant Borel-Moore homologies of the $G$-orbits in $X$. 
\end{Theorem}


In our second result, we make an additional assumption so that we have an
analogous result for the non-equivariant Borel-Moore homology groups.

\begin{Theorem}\label{I:T2}
Let $X$ and $G$ be as in Theorem~\ref{I:T1}, and let $T$ be a maximal torus of $G$.
If each $G$-orbit in $X$ contains a $T$-fixed point, then $X$ has zero cohomology in odd degree. 
Furthermore, the Borel-Moore homology of $X$ has a filtration 
as a graded $\Q$-vector space with associated graded 
space equal to the direct sum of Borel-Moore homologies of $G$-orbits. 
\end{Theorem}

A normal irreducible algebraic variety $X$ on which a connected reductive algebraic group $G$ acts 
with a morphism $G\times X \to X$ is called a {\em spherical $G$-variety} 
if a Borel subgroup $B$ of $G$ has an open orbit in $X$. 
It is well-known that a spherical variety $X$ has only finitely many $G$-orbits in it, see~\cite{LunaVust}. 
In fact, $X$ is a spherical $G$-variety if and only if there are only finitely many $B$-orbits in $X$,
see~\cite{Vinberg:Complexity,Brion86}. 
Therefore, Theorem~\ref{I:T1} is applicable to any spherical $G$-variety.
A particularly interesting subset of the set of spherical varieties consists of {\em wonderful varieties},
defined by the following properties: 
\begin{enumerate}
\item $X$ is a smooth projective $G$-variety with an open $G$-orbit $X_0$;
\item $X\setminus X_0 = \cup_{i=1}^l D_i$, where $D_i$'s are smooth prime 
$G$-divisors with normal crossings and $\cap_{i=1}^l D_i \neq \emptyset$; 
\item if $x$ and $x'$ are two points from $X$ such that 
$\{ i: x\in D_i \} = \{ i : x'\in D_i \}$, then $G\cdot x=G\cdot x'$.
\end{enumerate}
The divisors $D_i$ ($1\leq i \leq l$) are called the {\em boundary divisors}.
Note that the intersection of all boundary divisors is a closed subset of $X$,
therefore, it is isomorphic to a flag variety $G/Q$ for some parabolic subgroup $Q$ in $G$.


Our third result is an application of the previous theorem. 
\begin{Theorem}\label{I:T3}
Let $H$ be a closed subgroup of $G$ such that $G/H$ is isomorphic to the open $G$-orbit 
in a wonderful variety $X$. 
Let \hbox{$\{ Y_0 \cong G/H,Y_1,\dots, Y_{r}\}$} be the set of all $G$-orbits in $X$,
where the closed orbit, $Y_r$ is isomorphic to the flag variety $Y_{r} \cong G/Q$ 
for some parabolic subgroup $Q$. 
If the ranks of $G$ and $H$ are equal, 
then there is a filtration of the Borel-Moore homology $H_*(X,\Q)$, 
as a graded module over $\Q$, with the associated
graded module equal to the direct sum of Borel-Moore homology groups of the $G$-orbits in $X$.
\end{Theorem}

Let $G$ be a connected reductive group, let $T$ be a maximal torus in $G$,
and let $B$ be Borel subgroup such that $T\subset G$. 
Let $Q$ be a parabolic subgroup containing $B$. 
It is well-known that the cohomology ring of a (partial) flag variety $G/Q$ is isomorphic 
to the $W_Q$-invariants in the ring $H^*(G/B)$, where $W_Q$ is the Weyl group of (a Levi factor of) $Q$.
Another application of our theorem concerns the tensor products of vector spaces of the form $H^*(G/Q,\Q)$.
Let $W$ denote the Weyl group, $N_G(T)/T$, and 
let $\Delta$ denote the set of simple roots determined by $(B,T)$.
Let $I$ and $J$ be two subsets from $\Delta$. We denote by $P_I$ (resp. by $P_J$)
the parabolic subgroup corresponding to $I$ (resp. to $J$), and we denote by 
$W_I$ (resp. by $W_J$) the Weyl group of $P_I$ (resp. of $P_J$). 
The Weyl groups $W_I$ and $W_J$ are subgroups of $W$. 
For $w$ in $W$, let us denote by $\dot{w}$ an element of $N_G(T)$ 
representing $w$.  
Finally, let us denote by ${}^IW^J$ the subset of $W$ consisting of element $w\in W$ such that if $v\in W_I w W_J$, then 
the minimum number of simple reflections required to write $v$ as a product is greater than 
or equal to that of $w$.

\begin{Theorem}\label{I:T4}
Let $P_I$ and $P_J$ be two standard parabolic subgroups of a connected reductive subgroup $G$,
and let $X$ denote the double flag variety $X:=G/P_I\times G/P_J$ on which $G$ acts diagonally. 
Then every $G$-orbit in $X$ contains a $T$-fixed point, and furthermore, 
as a $\Q$-vector space, the cohomology ring of $X$ decomposes into a direct sum of cohomology rings as follows:
\hbox{$H^*(G/P_I\times G/P_J,\Q) \cong \bigoplus_{w\in{}^IW^J} H^*(G/\textrm{Stab}_G(\dot{w}),\Q)$}.
\end{Theorem}

\vspace{.25cm}
The structure of our paper is as follows.
In the preliminaries section, we summarize some basic properties 
of the equvariant Chow groups and the equivariant Borel-Moore homology groups. 
In Sections~\ref{S:3}--\ref{S:5}, we will prove the theorems that are stated above.

\vspace{.25cm}
\textbf{Acknowledgements.} The second author is partially supported
by a grant from Louisiana Board of Regents.
We thank Slawomir Kwasik and \"Ozlem U\v{g}urlu for useful conversations.
We are grateful to Michel Brion for his constructive criticisms and comments. 
Finally, we thank the referee her/his comments and suggestions which improved the quality of our paper.

\section{Notation and Preliminaries}\label{S:Preliminaries}

Throughout this paper, $G$ will stand for a connected algebraic group defined over $\C$. 
We will often assume that $G$ is reductive, and we will use the letter $B$ to denote a Borel subgroup of $G$.
We will use $T$ to denote a maximal torus that is contained in $B$. 
Then $B = T \ltimes U$, where $U$ is the unipotent radical of $B$. 
We will denote the Weyl group of $(G,T)$ by the letter $W$.
Then $W= N_G(T)/T$, where $N_G(T)$ is the normalizer of $T$ in $G$.

\subsection{(Equivariant) Chow Groups}

The equivariant Chow groups and the equivariant Borel-Moore homology 
of a $G$-space are introduced by Edidin and Graham in~\cite{EdidinGraham}.
In~\cite{BrionChow}, by focusing on the torus actions, Brion constructed 
machinery for the applications of equivariant Chow groups. Among the
results of Brion is a presentation of 
the equivariant Chow groups via invariant cycles, which establishes 
a connection between the equivariant Chow groups and the ordinary Chow groups. 
We will provide more details for this fact in the sequel.

To setup our notation, we will briefly review the definitions and basic 
properties of the equivariant Chow groups and Borel-Moore homology groups. 
We will follow the notation of~\cite{Brion98}, however, we will not abbreviate $H^G_*(X,\Q)$ to $H^G(X)$.

Let $X$ be a complex algebraic $G$-scheme, where $G$ is a connected complex linear algebraic group $G$. 
Notice that we do not assume the reductiveness of $G$, however, implicit in the definition of a $G$-scheme
is the assumption that it has a covering by $G$-invariant quasi-projective open sets.  
Let $n$ be a nonnegative integer. 
To define the $k$-th $G$-equivariant Chow group of $X$, we fix a linear representation 
$V$ of $G$, and we fix a $G$-stable open subset $U$ in $V$ satisfying the following 
two conditions: 
\begin{enumerate}
\item the quotient $U\rightarrow U/G$ exists, and it is a principal $G$-bundle;
\item $k \leq \text{codim}_V (V\setminus U)$.
\end{enumerate}
Then the quotient of the product $X\times U$ by the diagonal action of $G$ exists as 
a scheme, which is denoted by $X\times_G U$. 
If $X$ is equidimensional, then the {\em $G$-equivariant Chow group of degree $k$}
of $X$ is defined by 
\begin{align}\label{A:Def1}
A^k_G (X):= A^k ( X\times_G U).
\end{align}
If $X$ is not necessarily equidimensional, then we put 
\begin{align}\label{A:Def Chow2}
A^G_k (X):= A_{k-\dim (G) + \dim (U)} ( X\times_G U).
\end{align}
Obviously, if $X$ is equidimensional, then $A^k_G (X) = A^G_{\dim (X)-k} (X)$.

The $k$-th $G$-equivariant Borel-Moore homology of $X$ is defined by 
$$
H^G_{k} (X):= H_{k-2\dim (G) + 2\dim (U)} (X\times_G U).
$$
The rational equivariant Borel-Moore homology groups as well as 
the rational equivariant Chow groups are defined, as in the ordinary
case, by tensoring the corresponding groups with $\Q$.

\begin{Remark}\label{R}

We will mention several important results that are of crucial importance for our purposes. 

\begin{enumerate}


\item The (degree doubling) cycle map between the ordinary Chow groups and
Borel-Moore groups lifts to the equivariant setup, $cl_X^G: A^G_k(X)_\Q \rightarrow H^G_{2k}(X,\Q)$.

\item Let $G$ be a connected reductive complex algebraic group, 
and let $S$ denote the symmetric algebra over the rationals of 
the character group of a maximal torus $T$ of $G$,
We denote by $W$ the Weyl group of $(G,T)$. In this notation, 
we have the equality $A^*_G(pt)_\Q = S^W$, see~\cite[Theorem 10]{Brion98}.
Also, we see from Lemma~\ref{L:Brion's C12} that 
if $P$ is a parabolic subgroup of $G$ with a Levi subgroup $L\subset P$
such that $T\subset L$, then $A^*(G/P)_\Q = S^{W_L}/(S^W_+)$, where $W_L$ is 
the Weyl group of the pair $(L,T)$.
Since $W_L\subset W$, we see that $A^*(G/P)_\Q$ has the structure of a $A^*_G(pt)_\Q$-module.
In fact, for any $G$-scheme $X$, 
the $G$-equivariant Chow group $A^G_*(X)$ is a graded module over $A^G_*(pt)$.
These statements remain true if we replace the equivariant Chow rings with 
the equivariant cohomology rings.

\item Let $X$ be a $G$-scheme, where $G$ is as in the previous item. 
It is shown in~\cite[Proposition 11]{Brion98} that there is a canonical isomorphism
$\rho_G : A^G_*(X)_\Q / (S^W_+) \rightarrow A_*(X)_\Q$.

\item 
If $Y$ is a homogeneous variety of the form $G/K$ for some closed subgroup $K$ of $G$, 
then the equivariant cycle map $cl^G_Y : A^G_*(Y)_\Q \rightarrow H^G_*(Y,\Q)$ 
is an isomorphism. Indeed, it is easy to see from definitions, by using the homotopy invariance property, that 
$A^G_*(G/K)_\Q=A^K_*(pt)_\Q= A_*(BK)_\Q$ and that  
$H^G_*(G/K,\Q)=H^K_*(pt,\Q)= H_*(BK,\Q)$. Here, $BK$ denotes the classifying space of $K$. 
The ordinary cycle map $cl_{BK} : A_*(BK)_\Q \rightarrow H_*(BK,\Q)$
is an isomorphism. This fact is proven by Edidin and Graham in~\cite{EdidinGraham0}
by building on Totaro's work~\cite{TotaroChow}. Since the cycle map is degree doubling, 
it also follows from this analysis that $Y$ has zero equivariant Borel-Moore homology in 
odd degrees.

\end{enumerate}

\end{Remark}

\section{Proof of Theorem~\ref{I:T1}}\label{S:3}

\begin{proof}[Proof of Theorem~\ref{I:T1}]

We start with a quick reduction argument since we do not assume 
that $G$ is reductive. Let $R_u(G)$ denote the unipotent radical of $G$,
and let $G'$ denote the reductive group $G/R_u(G)$. 
Then, for an open set $U$ as in the definition of the 
$G$-equivariant Chow groups, see (\ref{A:Def1}), we consider the quotient 
$U \to U/G'$. This quotient exists, and furthermore, it is a principal $G'$-bundle. 
Thus we get a smooth map $X\times_{G'} U \to X\times_G U$ with fibers $G/G'$.
But the fiber $G/G'$ is an affine space, therefore, $A^G_*(X) = A^{G'}_*(X)$. 
Also by the same argument, we see that $H^G_*(X,\Q)  = H^{G'}_*(X,\Q)$.
From now on we will assume that $G$ is a reductive group.

Let $Y$ be a closed $G$-orbit, and let $U$ denote its complement in $X$. 
We proceed to analyze the equivariant cycle maps for $U$ and $X$.
We claim that these maps are isomorphisms. 
Notice that the number of $G$-orbits in $U$ is one less than the
number of $G$-orbits in $X$, so, we will use induction on the number of orbits. 
The base case, where there is a single $G$-orbit, is taken care of
by Remark~\ref{R}. 
To finish the inductive argument, we will use the $G$-equivariant localization diagram for $(Y,X,U)$:
\begin{figure}[h]
\begin{center}
\scalebox{.85}{
\begin{tikzpicture}
\node at (-1.25,1) (a) {$A_k^G(X)_\Q$};
\node at (-4.25,1) (a') {$A_k^G(Y)_\Q$};

\node at (2.25,1) (b) {$A_k^G(U)_\Q$};
\node at (5.25,1) (b') {$0$};
\node at (-1.25,-1) (c) {$H_{2k}^G(X,\Q)$};
\node at (-4.25,-1) (c') {$H_{2k}^G(Y,\Q)$};
\node at (-7.25,-1) (e') {$H_{2k+1}^G(U,\Q)$};
\node at (-7.25,1) (e) {$A_{k+1}^G(U,1)_\Q$};
\node at (2.25,-1) (d) {$H_{2k}^G(U,\Q)$}; 
\node at (5.25,-1) (d') {$H_{2k-1}^G(Y,\Q)$};
\node at (-3.75,0) {$cl_Y^G$}; 
\node at (-.75,0) {$cl_X^G$}; 
\node at (2.75,0) {$cl_U^G$}; 
\node at (0.5,1.35) {$i^*$}; 
\node at (0.5,-.65) {$i^*$}; 
\node at (-2.75,1.35) {$j_*$}; 
\node at (-2.75,-.65) {$j_*$}; 
\node at (3.5,-.65) {$\delta_{2k}$}; 
\node at (-5.75,-.65) {$\delta_{2k+1}$}; 
\draw[->,thick] (a) to (b);
\draw[->,thick] (a') to (a);
\draw[->,thick] (a') to (c');
\draw[->,thick] (c') to (c);
\draw[->,thick] (a) to (c);
\draw[->,thick] (b) to (d);
\draw[->,thick] (b) to (b');
\draw[->,thick] (d) to (d');
\draw[->,thick] (c) to (d);
\draw[->,thick] (e') to (c');

\draw[->,thick] (e) to (a');
\end{tikzpicture}
}
\end{center}
\label{F:commutative}
\caption{The $G$-equivariant localization diagram for $(Y,X,U)$.}
\end{figure}

From the commuting of the squares in the diagram, we see 
that the long exact sequence of equivariant (higher) Chow groups
split at the third term, and the long exact sequence of 
equivariant Borel-Moore homology groups splits into short exact sequences.
In particular, we see from these splittings the isomorphisms 
\begin{align}\label{A:second}
A^G_*(X)_\Q \cong H^G_*(X,\Q) \text{ and } \ 
H^G_*(X,\Q) \cong H^G_*(Y,\Q) \oplus H^G_*(U,\Q).
\end{align}
If $H$ is a closed subgroup of $G$, then the $H$-equivariant 
Borel-Moore homology of a point has the structure of a (graded) $H^G_*(pt,\Q)$-module,
therefore, for every $G$-orbit $Y'$ in $X$, the same statement holds true for $H^G(Y',\Q)$.   
Notice that we can re-iterate these arguments for $H^G_*(U,\Q)$ as well. 
Therefore, in light of the second isomorphism in (\ref{A:second}), we see that 
$H^G_*(X,\Q)$ has a filtration, as a graded $H^G_*(pt,\Q)$-module, given by the sums of the equivariant
Borel-Moore homology groups of the $G$-orbits.
Clearly, the associated graded module of this filtration is the direct sum of all $G$-equivariant Borel-Moore homology
groups of the orbits of $X$. This finishes the proof.
\end{proof}

\section{Proof of Theorem~\ref{I:T2}}\label{S:4}

Our assumptions on $X$ and $G$ are as before; 
$X$ is a complex algebraic variety on which a connected complex linear 
algebraic group $G$ acts morphically with finitely many orbits. 
In addition, we assume the existence of a maximal torus $T$ of $G$ such that 
in every $G$-orbit in $X$, there exists a $T$-fixed point.
Therefore, if $Z$ is a $G$-orbit in $X$, then $Z\cong G/K$ 
for some closed subgroup $K$ with $T\subset K$. In particular, 
$\mt{rk} (K) = \mt{rk} (G)$, hence, the odd rational cohomology of $Z$ 
is zero by~\cite[Corollary 12]{Brion98}. Below is a strengthening of this result. 

\begin{Lemma}\label{L}
If $Z$ denotes a $G$-orbit in $X$, then $Z$ has zero cohomology in odd degrees. 
\end{Lemma}
\begin{proof}
By our assumption, 
$Z$ is of the form $G/K$ for some closed subgroup $K$ with $T\subset K$. 
Observe that, to prove our claim, it suffices to prove it for $G/K^0$, where 
$K^0$ is the connected component of $K$ containing the identity. 
Indeed, $G/K$ is the quotient of $G/K^0$ by the finite group $K/K^0$. 

Now, since $T$ is connected, $T\subset K^0$, hence, we have a fibration 
$K^0/T \to G/T \to G/K^0$.
It is not difficult to see that the cohomology ring of $K^0/T$ is equal to the 
cohomology ring of the full flag variety of $K^0$, which is 
well-known to be generated by the Chern classes of line bundles associated 
with characters of $T$. In particular we see that these line bundles 
on $K^0/T$ extend to $G/T$, hence we can apply the Leray-Hirsch 
theorem to conclude that 
$$
H^*(G/T) \cong H^*(K^0/T) \otimes H^*(G/K^0)
$$
as graded $\Z$-modules. 
Since the cohomology of $G/T$ has zero cohomology 
in odd degrees, the same statement holds true for 
the cohomologies of $K^0/T$ and $G/K^0$. This finishes the proof.

\end{proof}

\begin{proof}[Proof of Theorem~\ref{I:T2}]

We continue with the notation of Lemma~\ref{L}, that is, 
$Z=G/K$ with $T\subset K$. Since all odd cohomology groups of $Z$ are zero,
the Leray spectral sequence of the universal bundle 
$EG \times_G Z \to BG$
degenerates at the second page, hence, the equivariant cohomology of $Z$ 
is a free module over $H^*(BG,\Q) = H^*_G(pt,\Q)$. 
As this property holds true for every $G$-orbit in $X$, we conclude that 
$H^G_*(X,\Q)$ is a free module over $H^*_G(pt,\Q)$. 
Thus, by tensoring the filtration (from Theorem~\ref{I:T1}) of the equivariant Borel-Moore homology of $X$ with 
$\Q = H^*_G(pt,\Q)/ (S^W_+)$, we finish the proof of our theorem.

\end{proof}

\section{Applications}\label{S:5}

\subsection{Proof of Theorem~\ref{I:T3}}\label{SS:5}

Let $X$ be an irreducible $G$-scheme, and let $X_0$ denote the open $G$-orbit in $X$. 
We will call $X$ {\em CH-congruent} (Chow group/rational homology congruent) if 
the cycle map to the rational cohomology induced via the Poincar\'e duality, that is 
$cl_{X_0} : A^* (X_0)_\Q \rightarrow H^*(X_0,\Q)$, 
is an isomorphism. 
More generally, we will call a $G$-scheme $X$ {\em CH-congruent} 
if each irreducible component of $X$ is CH-congruent.

Let $H$ be a connected reductive subgroup of a connected reductive group $G$.
Let $T_H$ be a maximal torus of $H$ and let $T$ be a maximal torus of $G$ containing $T_H$.
We will denote by $W_H$ and $W$ the Weyl groups of $(H,T_H)$ and $(G,T)$, respectively. 
We denote by $S_H$ and $S$ the symmetric $\Q$-algebras over the character groups of $T_H$ and $T$ 
respectively. Finally, we will denote by $(S_+^W)$ 
the ideal generated by the restriction to $S^{W_H}_H$ of the homogeneous positive degree $W$-invariants in $S$.

\begin{Lemma}\label{L:Brion's C12}
Let $H$ be a closed, connected, and reductive subgroup of a connected reductive group $G$. 
Then the following are equivalent:
\begin{enumerate}
\item The rational cohomology of $G/H$ vanishes in odd degree.
\item The ranks of $G$ and $H$ are the same.
\item The homogeneous space $G/H$ is CH-congruent.
\end{enumerate}
Moreover, 
$
A^*(G/H)_\Q \cong S_H^{W_H}/ (S_+^W) \ \text{  and }\   
A^*_T(G/H)_\Q \cong S \otimes_{S^W} (S_+^W).
$
\end{Lemma}
\begin{proof}
See~\cite[Corollary 12]{Brion98}.
\end{proof}

As before, for every closed imbedding $j: Y\hookrightarrow X$ 
and the complementary imbedding $i: U\hookrightarrow X$, where $U:=X\setminus Y$, 
we have a localization diagram as in Figure~\ref{F:commutative2}: 

\begin{figure}[h]
\begin{center}
\scalebox{.85}{
\begin{tikzpicture}
\node at (-1.25,1) (a) {$A_k(X)_\Q$};
\node at (-4.25,1) (a') {$A_k(Y)_\Q$};

\node at (2.25,1) (b) {$A_k(U)_\Q$};
\node at (5.25,1) (b') {$0$};
\node at (-1.25,-1) (c) {$H_{2k}(X,\Q)$};
\node at (-4.25,-1) (c') {$H_{2k}(Y,\Q)$};
\node at (-7.25,-1) (e') {$H_{2k+1}(U,\Q)$};
\node at (-7.25,1) (e) {$A_{k+1}(U,1)_\Q$};
\node at (2.25,-1) (d) {$H_{2k}(U,\Q)$}; 
\node at (5.25,-1) (d') {$H_{2k-1}(Y,\Q)$};
\node at (-3.75,0) {$cl_Y$}; 
\node at (-.75,0) {$cl_X$}; 
\node at (2.75,0) {$cl_U$}; 
\node at (0.5,1.35) {$i^*$}; 
\node at (0.5,-.65) {$i^*$}; 
\node at (-2.75,1.35) {$j_*$}; 
\node at (-2.75,-.65) {$j_*$}; 
\node at (3.5,-.65) {$\delta_{2k}$}; 
\node at (-5.75,-.65) {$\delta_{2k+1}$}; 
\draw[->,thick] (a) to (b);
\draw[->,thick] (a') to (a);
\draw[->,thick] (a') to (c');
\draw[->,thick] (c') to (c);
\draw[->,thick] (a) to (c);
\draw[->,thick] (b) to (d);
\draw[->,thick] (b) to (b');
\draw[->,thick] (d) to (d');
\draw[->,thick] (c) to (d);
\draw[->,thick] (e') to (c');

\draw[->,thick] (e) to (a');
\end{tikzpicture}
}
\end{center}
\label{F:commutative2}
\caption{The nonequivariant localization diagram for $(Y,X,U)$.}
\end{figure}

We will make use of the above (nonequivariant) diagram in the context of equivariant embeddings 
of CH-congruent spaces. 
Let $X$ be a smooth complete spherical embedding of a 
CH-congruent homogeneous space $G/H$. 
Under these assumptions, we know that $X$ has a cellular decomposition via 
Bia{\l}ynicki-Birula decomposition. 
Let $Y$ denote the complement of $U:=G/H$ in $X$,
let $i: U\hookrightarrow X$ and $j : Y \rightarrow X$ denote the corresponding inclusions.
Since $U$ is a CH-congruent homogeneous space, 
the long exact sequences in Figure~\ref{F:commutative2} break down into short 
exact sequences. 
Indeed, the vertical arrows, $cl_U$ and $cl_X$ are isomorphisms.
On both of the top and the bottom sequences, 
the maps $i^*$'s are induced from the restriction map 
to the open set $U$, and therefore, they are actually equivalent maps.
(So, in this case, our notation is unambiguous.) It follows that, 
by the exactness of the Chow group sequence, $i^*$ is 
surjective, hence $\delta_{2k}$ is the 0-map. 
Note that $U$ is a CH-congruent variety, so, its odd (co)homology 
vanishes, therefore, $\delta_{2k+1}$ is the 0 map as well.
It follows that $H_{2k+1} (Y,\Q) \cong H_{2k+1}(X,\Q)$.

By the exactness of the Chow sequence,
and by identifying $i^*$'s in the 
top and the bottom sequences, we see that 
the images of $j_*$'s are isomorphic subgroups.
Furthermore, we know that the bottom $j_*$ is injective since $\delta_{2k+1}=0$.
However, we do not claim that the $j_*$'s are 
the `same' (yet) since we do not know if $cl_Y$ is an 
isomorphism, or not. 
It turns out that this is the case by a deep result due to Totaro and Jannsen. 

According to~\cite{Totaro14}, a {\em linear scheme} is a scheme which can be obtained 
by an inductive procedure starting with an affine space of any dimension, in such a way that 
the complement of a linear scheme imbedded in affine space is also a linear scheme,
and a scheme which can be stratified as a finite disjoint union of linear schemes is a linear scheme. 

\begin{Lemma}\label{L:TotaroJannsen}
For any complex linear scheme $Z$, the natural map
$$
A_i(Z)_\Q \longrightarrow W_{-2i} H_{2i} (Z,\Q) 
$$ 
from the rational Chow group into the smallest subspace of 
rational Borel-Moore homology with respect to the weight filtration 
is an isomorphism.
\end{Lemma}

\begin{proof}
See~\cite[Theorem 3]{Totaro14}.
\end{proof}

\begin{Remark}\label{R:general remark}
\begin{enumerate}
\item A finite union of linear schemes is a linear scheme. 
\item Any spherical variety is linear, see the addendum at the end of Section 2 in~\cite{Totaro14}.
\item The complement of the open $G$-orbit in a spherical variety is a linear scheme since 
it is a union of finitely many $G$-orbit closures.
\end{enumerate}
\end{Remark}

Now, by Remark~\ref{R:general remark} and Lemma~\ref{L:TotaroJannsen}, 
we know that the vertical map $cl_Y$ is injective. Since the bottom 
$j_*$ is injective, and it has the same image as the top 
$j_*$, we see that the top $j_*$ has to be injective as well. 
It follows that the localization long exact sequence for 
the higher Chow groups breaks down. These observations 
give us the following result.

\begin{Theorem}\label{T:break it}
Let $X$ be a smooth complete spherical $G$-variety. 
If the open $G$-orbit, $G/H$ is CH-congruent, then 
we have 
$$
H_{2k}(X-G/H,\Q) \cong A_k(X-G/H)_\Q \ \quad \text{for all } k\geq 0,
$$
and 
$$
A_k(X)_\Q \cong A_k(X-G/H)_\Q \oplus A_k(G/H)_\Q \ \quad \text{for all } k\geq 0.
$$
Furthermore, we have 
$H_{2k+1}(X,\Q) = H_{2k+1}(X-G/H,\Q) = \{ 0 \}\ \quad \text{for all } k\geq 0$.
\end{Theorem}

\begin{proof}
Let $U$ denote the open $G$-orbit, $U=G/H$, and let $Y$ 
denote its complement, $X-U$. By the above discussion, 
it suffices to show the last assertion that, for all $k \in \Z_{\geq 0}$, the groups 
$H_{2k+1} (Y,\Q) \cong H_{2k+1}(X,\Q)$ are trivial.
But this follows from the fact that 
the Borel-Moore homology is equal to the singular homology for an oriented compact manifold,
and the fact that $X$ has a cellular decomposition, 
where the cells are complex affine spaces.
Indeed, $X$ has a Bia{\l}ynicki-Birula decomposition since it is a smooth complete 
spherical variety, hence the odd homology of $X$ vanishes.
\end{proof}

Let $H$ be a closed subgroup of $G$
such that the subset $BH$ is open in $G$. In this case, $H$ is called 
a {\em spherical subgroup} and the homogeneous space $G/H$ is called a 
{\em spherical homogeneous space}.

Let $G/H$ be a spherical homogeneous space. 
A spherical embedding $X$ of $G/H$ is called {\em toroidal} if $X$ has
no $B$-invariant divisor which is $G$-invariant. 
$X$ is called {\em simple} if it has a unique closed $G$-orbit. 
Wonderful varieties are toroidal as shown by Bien and Brion in~\cite{BienBrion}.
In fact, any smooth complete toroidal spherical variety is a wonderful variety. 
The stabilizers of generic points in $G$-orbits of a simple toroidal embedding 
are recently looked at by Batyrev and Moreau in~\cite{BatyrevMoreau}.

\begin{Lemma}\label{L:BM}
Let $X$ be a simple toroidal spherical embedding of $G/H$ 
corresponding to an uncolored cone $\sigma$. 
Denote by $I(\sigma)$ the set of all spherical roots in $\mathcal{S}$ 
that vanish on $\sigma$.
Let $x'$ be a point in the unique closed $G$-orbit $X'$ of $X$, and 
$\mt{Iso}_G(x')$ the isotropy subgroup of $x'$ in $G$. Then, up to a 
conjugation, we have the inclusions:
$$
H_{I(\sigma)} \subset \mt{Iso}_G(x')\subset N_G(H_{I(\sigma)}).
$$
Moreover, there is a homomorphism from $N_G(H_{I(\sigma)})$ to $\mt{Hom}(\sigma^\perp \cap M, \C^*)$
whose kernel is $\mt{Iso}_G(x')$.
\end{Lemma}
\begin{proof}
See~\cite[Theorem 1.2]{BatyrevMoreau}.
\end{proof}

The subgroups $H_{I(\sigma)}$ in Lemma~\ref{L:BM} are called the {\em satellites of $H$}. 
For a closed subgroup $K$ of $G$, let $K^{red}$ denote the maximal reductive subgroup of $K$. 

\begin{Remark}\label{R:toroidal}
In~\cite[Lemma 7.2]{BatyrevMoreau}, Batyrev and Moreau proved the following result,
which they attribute to Brion and Peyre~\cite{BrionPeyre}:
If $I$ is a subset of the spherical roots $\mathcal{S}$ of $G/H$, then $H_I^{red} \subset H$. 
\end{Remark}

\begin{Proposition}\label{P:break it}
Let $H$ be a spherical subgroup of $G$ of maximal rank, that is $\mt{rk} (H)=\mt{rk} (G)$. 
We assume that $G/H$ has a smooth complete toroidal embedding, denoted by $X$. 
Let $X_i$ be a $G$-orbit closure in $X$.
If $H_i$ is the stabilizer of a generic point in $X_i$, that is $H_i = \mt{Stab}_G(x)$
for some $x$ in the open $G$-orbit in $X_i$, then the rank of $H_i$ is equal to rank of $G$. 
Furthermore, we have the isomorphisms 
$H_k(X_i-G/H_i,\Q) \cong A_k(X_i-G/H_i)_\Q$ and 
$A_k(X_i)_\Q \cong A_k(X_i-G/H_i)_\Q \oplus A_k(G/H_i)_\Q$ for all $k\geq 0$.
\end{Proposition}

\begin{proof}
Since $X$ is a wonderful variety, any $G$-orbit closure $X_i$ is smooth,
and furthermore, $X_i$ is a wonderful variety itself. 
Therefore, in light of Theorem~\ref{T:break it}, 
our second claim will follow from the first one,
that is, 
the rank of the stabilizer of a generic point in $X_i$ is equal to $\mt{rk} (H) = \mt{rk} (G)$. 

Note that the closed $G$-orbit in $X_i$ is the closed $G$-orbit in $X$. 
Let $G/Q$ denote this closed $G$-orbit, where $Q$ is a parabolic subgroup. 
The satellite of $H$ corresponding to $G/Q$ is $H_\emptyset = Q$. 
It follows from Remark~\ref{R:toroidal} that the Levi subgroup $L_Q$ of $Q$ 
is contained in the stabilizer subgroup in $G$ of a generic point in $X_i$. 
Since the rank $\mt{rk} (L_Q)$ of $L_Q$ is equal to the rank of $G$, the proof of 
our proposition is complete. 
\end{proof}

We are now ready to prove the third result that we announced in the introduction.

\begin{proof}[Proof of Theorem~\ref{I:T3}]

We will apply Theorem~\ref{I:T2}. In the course of the proof of Proposition~\ref{P:break it}
we showed that for every $G$-orbit $Y_i$ in $X$, the stabilizer of a generic point from $Y_i$
contains the Levi subgroup $L_Q$ of the parabolic subgroup $Q$, where $G/Q$ is the closed orbit in $X$. 
Since $T\subset L_Q$, we see that every $G$-orbit in $X$ has a $T$-fixed point.
This finishes the proof. 

\end{proof}

\subsection{Proof of Theorem~\ref{I:T4}}\label{SS:6}

Let $G$ be a connected reductive algebraic group, let $T$ and $B$ denote a maximal torus 
and a Borel subgroup such that $T\subset B \subset G$. 
Let $I$ and $J$ be two subsets from the set of simple roots $\Delta = \Delta(G,B,T)$. 
A {\em double flag variety} is a product of the form $X:=G/P_I\times G/P_J$,
where $P_I$ and $P_J$ respectively are the standard parabolic subgroups determined by the subsets $I$ and $J$.
Starting with the work of Littelmann in~\cite{Littelmann}, it is understood that the cones over these varieties can be used 
for calculating the multiplicities in the tensor products of irreducible representations of $G$. For a good exposition of this idea, see~\cite[Section 2.11]{Timashev}.

The diagonal $G$-action on a double flag variety $X$ is given by $g\cdot (aP_I,bP_J)= (gaP_I,gbP_J)$ for $g,a,b\in G$.
The canonical projection $\pi : X\rightarrow G/P_J$
is $G$-equivariant and it turns $X$ into a 
homogenous fiber bundle over $G/P_J$ with fiber $G/P_I$
at every point $gP_J$ ($g\in G$)
of the base $G/P_J$. 
To distinguish it from the other fibers,
let us denote by $Y$ the fiber $G/P_I$ 
at the `origin' $eP_J$ of $G/P_J$. Then 
any $G$-orbit in $X$ meets $Y$.
Note also that if $g\cdot y \in Y$
for some $g\in G$ and $y\in X$, then 
$g\in P_I$. 
This means that, for every $G$-orbit in $X$, there exists a unique $P_I$-orbit in $G/P_J$, which is 
viewed as a fiber of $\pi$ at $eP_J$.  
It is well-known that these $P_I$-orbits are in 1-1 correspondence with $(W_I,W_J)$-double cosets in $W$,
where $W_I$ and $W_J$, respectively, are the subgroups of $W$ determined by $I$ and $J$, see~\cite[Section 21.16]{Borel}.
In fact, every $P_I$-orbit in $G/P_J$ is of the form $P_I\cdot \dot{w}$, where $\cdot{w} \in N_G(T)$
and $w\in W$ is a minimal length representative of a double 
coset $W_I w W_J$ in $W$.

Let us denote by ${}^IW^J$ the set of minimal coset representatives of the $(W_I,W_J)$-double cosets in $W$. 
It follows from the above discussion that the set of $G$-orbits in $X$ is finite as it is in bijection with ${}^IW^J$.
Furthermore, each $G$-orbit in $X$ is of the form $G\cdot \dot{w} \cong G/\textrm{Stab}_G(\dot{w})$ for some  
$w\in {}^IW^J$.
Note that the stabilizer subgroup of any element of $W$ is a standard parabolic subgroup in $G$, therefore, 
\hbox{$T\subset \textrm{Stab}_G(\dot{w})$}.

We are now ready to prove the final result of our paper.

\begin{proof}[Proof of Theorem~\ref{I:T4}.]
We already know that the $G$-orbits in the double flag variety $X=G/P_I\times G/P_J$ are parametrized by ${}^IW^J$.
It follows from the above discussion that every $G$-orbit in $X$ contains a $T$-fixed point.
By the K\"unneth formula, the (integral) cohomology ring of $X$ is isomorphic to the tensor product of the 
(integra) cohomology rings $H^*(G/P_I)$ and $H^*(G/P_J)$. 
Therefore, after tensoring with $\Q$, Theorem~\ref{I:T2} gives us the $\Q$-vector space decomposition in 
$H^*(G/P_I,\Q)\otimes H^*(G/P_J,\Q) \cong \bigoplus_{w\in{}^IW^J} H^*(G/\textrm{Stab}_G(\dot{w}),\Q)$.
\end{proof}

The {\em complexity} of a group action $G\times X\to X$ is defined as the minimal codimension of a Borel orbit in general position. 
A $G$-variety $X$ is spherical if and only if $X$ is normal and the complexity of the action is zero, see~\cite[Section 3.15]{Timashev}.
For $G:=\mathbf{SL}_n$, the structures of the posets of $G$-orbit closures in the double flag varieties of complexity $\leq 1$ 
are determined in~\cite{Can18,CanTien}. This information can be used for further studying the decompositions described in Theorem~\ref{I:T4}.

\bibliography{References}
\bibliographystyle{plain}

\end{document}